\documentclass{amsart}
\usepackage{amssymb,latexsym,enumitem,graphics}
\pdfoutput=1

\theoremstyle{plain}
\newtheorem{thm}{Theorem}

\newtheorem{prop}[thm]{Proposition}
\newtheorem{conj}[thm]{Conjecture}

\theoremstyle{definition}

\newtheorem{ex}[thm]{Example}

\theoremstyle{remark}

\numberwithin{equation}{section}

\begin{document}
\title[The James Function]
{The James Function}
\author[C. N. B. Hammond, W. P. Johnson, and S. J. Miller]{Christopher N. B. Hammond, Warren P. Johnson, and Steven J. Miller}
\thanks{The third-named author was partially supported by NSF grant DMS1265673.}
\date{October 23, 2014}
\address{Department of Mathematics\\
Connecticut College\\
New London, CT 06320}
\email{cnham@conncoll.edu}
\email{wpjoh@conncoll.edu}
\address{Department of Mathematics and Statistics\\
Williams College\\
Williamstown, MA 01267}
\email{sjm1@williams.edu}
 
\maketitle
\thispagestyle{empty}

\section{Introduction}\label{S:intro}

In his 1981 \textit{Baseball Abstract} \cite{james}, Bill James posed the following problem:
suppose two teams $A$ and $B$ have winning percentages $a$ and $b$ respectively, having played
equally strong schedules in a game such as baseball where there are
no ties.  If $A$ and $B$ play each other, what is the probability $p(a,b)$ that $A$ wins?

This question is perhaps more relevant to other sports, because in baseball the outcome is particularly
sensitive to the pitching matchup.  (In 1972, the Philadelphia Phillies won 29 of the 41 games 
started by Steve Carlton, and 30 of the 115 games started by their other pitchers.)   The
answer is quite interesting, even if its applicability is somewhat limited by the tacit assumption of uniformity.

For $0<a<1$ and $c>0$, define $q_{c}(a)$ by
\begin{equation}\label{E:talent}
a=\frac{q_{c}(a)}{q_{c}(a)+c}\text{.}
\end{equation}
James calls $q_{\frac12}(a)$ the log5 of $a$, and does not consider any other values of $c$.  Under the assumption of uniformity, he claims that $p(a,b)$ would be given by the function
\begin{equation}\label{E:log5}
P(a,b)=\frac{q_{\frac12}(a)}{q_{\frac12}(a)+q_{\frac12}(b)}\text{.}
\end{equation}
In this context, we take uniformity to mean that a team's likelihood of defeating another team is determined only by their winning percentages. For example, this assumption
ignores the impact of the starting pitchers and precludes the situation where one team has a tendency to do particularly well or particularly poorly against another team.

This technique is sometimes called the log5 method of calculating $p(a,b)$, although we will avoid using this name as
there is nothing obviously logarithmic about it.  
It is easy to see from (\ref{E:talent}) that
\begin{equation*}
q_{c}(a)=\frac{ca}{1-a}\text{.}
\end{equation*}
Substituting this expression into (\ref{E:log5}), we see that
\begin{equation}\label{jamesfunction}
P(a,b)=\frac{a(1-b)}{a(1-b)+b(1-a)}\text{,}
\end{equation}
not only for $c=\frac{1}{2}$ but for any positive $c$.
The explicit form of $P(a,b)$ was first given 
by Dallas Adams \cite{james}, who also christened it the \textit{James function}.
It makes sense to extend the James function to values of $a$ and $b$
in the set $\{0,1\}$, except when $a=b=0$ or $a=b=1$.
In these two cases, we would not have expected to be able to make predictions based on
winning percentages alone.  Moreover, both cases would be impossible if the two
teams had previously competed against each other.

\bigskip

James's procedure can be interpreted as a twofold application of the general method known as the \textit{Bradley--Terry model}
(or sometimes the \textit{Bradley--Terry--Luce model}).
If $A$ and $B$ have worths $w(A)$ and $w(B)$ respectively, the probability that $A$
is considered superior to $B$ is
\[
\pi(A,B)=\frac{w(A)}{w(A)+w(B)}\text{.}
\]
Despite the attribution of this model to Bradley and Terry \cite{bradleyterry} and to Luce \cite{luce},
the basic idea dates back to Zermelo \cite{zermelo}.  The question,
of course, is how to assign the ``right" measure for the worth of $A$ in a particular setting.  In chess, for instance, it is common to express
the worth of  a player as $10^{R_{A}/400}$, where $R_{A}$ denotes the player's Elo rating (see \cite{glickmanjones}).  
(The rating of chess players is the question in which Zermelo was originally interested.
Good \cite{good}, who also considered this problem, seems to have been
the first to call attention to Zermelo's paper.)
Another example is James's so-called Pythagorean model
(introduced in \cite[p.\ 104]{james0} and discussed further in \cite{miller1})
for estimating a team's seasonal winning percentage, based on the number $R$ of runs it scores and the number $S$ of runs it allows.
In this case, the worth of the team is $R^{2}$ and the worth of its opposition is $S^{2}$.

In the construction of the James function, we can view the measure of a team's worth as being obtained from the Bradley--Terry model itself.  We begin by
assigning an arbitrary worth $c>0$ (taken by James to be $\frac{1}{2}$) to a team with winning percentage $\frac{1}{2}$.  Equation (\ref{E:talent})
can be construed as an application of the Bradley--Terry model, where the worth of a team
is determined by the assumption that its overall winning percentage is equal to its probability of defeating a team with winning percentage $\frac{1}{2}$.
Equation (\ref{E:log5}) represents a second application of the Bradley--Terry model, where each team has an arbitrary winning percentage
and the measure of its worth comes from the previous application of the model.

This area of study, which is usually called the theory of paired comparisons,
has focused from the outset on the question of inferring worth from an incomplete set of outcomes \cite{zermelo}.  (See
\cite{david} for a thorough treatment,
as well as \cite{glickman} and \cite{stob} for additional context.)
James, on the other hand, takes the worths to be known and uses them to determine the probability of the outcomes.
We will adopt a similar point of view, emphasizing a set of axiomatic principles rather than a specific model.

\bigskip

James's justification \cite{james} for his method does not invoke the Bradley--Terry model,
but rather the fact that the resulting function $P(a,b)$ satisfies six self-evident conditions:

\begin{enumerate}
\item $P(a,a)=\frac12$.
\item $P(a,\frac12)=a$.
\item If $a>b$ then $P(a,b)>\frac12$, and if $a<b$ then $P(a,b)<\frac12$.
\item If $b<\frac12$ then $P(a,b)>a$, and if $b>\frac12$ then $P(a,b)<a$.
\item $0\le P(a,b)\le 1$, and if $0<a<1$ then $P(a,0)=1$ and $P(a,1)=0$.
\item $P(a,b)+P(b,a)=1$.
\end{enumerate}
Condition (1) pertains to the situation where two different teams have the same winning percentage
(as opposed to a single team competing against itself).  To avoid contradicting (5), condition (4) should exclude the cases where $a=0$ and $a=1$.
We will call this set, with this slight correction, the \textit{proto-James conditions}.
(James originally referred to them as ``conditions of logic.")
In addition to presenting some empirical evidence for (\ref{E:log5}), James makes
the following assertion.

\begin{conj}[1981]
The James function $P(a,b)$ is the only function that satisfies all six of the proto-James conditions.
\end{conj}

\noindent Jech \cite{jech} independently proposed a similar, albeit shorter list of conditions.  Although he did not consider the James conjecture,
he was able to prove a uniqueness theorem pertaining to a related class of functions.

The purpose of this paper is to examine the mathematical theory underlying the James function 
and to demonstrate that the James conjecture is actually false.  In fact, we will introduce
and study a large class of functions that satisfy the proto-James conditions.\bigskip

While the proto-James conditions are certainly worthy of attention, we prefer to work with a slightly different set.
The following conditions apply to all points $(a,b)$ with $0\leq a\leq 1$ and $0\leq b\leq 1$, except for $(0,0)$ and $(1,1)$: 

\begin{enumerate}[label=(\alph*)]
\item $P(a,\frac12)=a$.
\item $P(a,0)=1$ for $0<a\leq 1$.
\item $P(b,a)=1-P(a,b)$.
\item $P(1-b,1-a)=P(a,b)$.
\item $P(a,b)$ is a non-decreasing function of $a$ for $0\leq b\leq 1$ and a strictly increasing function of $a$ for $0<b<1$.
\end{enumerate}
We shall refer to conditions (a) to (e) as the \textit{James conditions}.  Condition (d), which
is not represented among the proto-James conditions, simply states that the whole theory could be reformulated using losing percentages 
rather than winning percentages, with the roles of the two teams reversed.  Together with condition (c),
it is equivalent to saying $P(1-a,1-b)=1-P(a,b)$, which may seem more natural to some readers.
It should be clear from (\ref{jamesfunction}) that the James function satisfies James conditions
(a) to (d).  We will verify condition (e) in Section \ref{S:adams}.
 
It is fairly obvious that the James conditions imply the proto-James conditions.
Condition (a) is identical to condition (2).  Condition (c) is condition (6), which implies (1) by taking $b=a$.  Condition 
(e) is stronger than (3) and (4), and in concert with (1) and (2) implies them both.  Combined with 
(c) or (d), it also implies that $P(a,b)$ is a non-increasing function of $b$ for $0\leq a\leq 1$ and a strictly decreasing function of $b$
for $0<a<1$.  Finally, (b) implies the second of the three parts of
(5).  Together with (c), it also implies that $P(0,b)=0$ if $0<b\leq 1$.  By taking 
$b=0$ in (d) and replacing $1-a$ with $b$, condition (b) further implies that $P(1,b)=1$ if $0\leq b<1$, and this
together with (c) gives $P(a,1)=0$ for $0\leq a<1$, which is (a hair stronger
than) the third part of (5).  These facts, combined with (e), show that $0<P(a,b)<1$ when $0<a<1$ and $0<b<1$, which implies the first part of (5).

We will focus our attention on functions that satisfy the James conditions, and hence also
the proto-James conditions.  See \cite{supplement}, the online supplement to this paper, for an example of a function that satisfies the proto-James conditions but not the James conditions.
 
\section{Verification of the James Function}\label{S:miller}

While the Bradley--Terry model is practically ubiquitous, its applicability to this situation is not obvious from an axiomatic perspective. We now present a self-contained proof that, under an intuitive probabilistic model in which $a$ and $b$ are the probabilities of success in simultaneous Bernoulli trials, the James function $P(a,b)$ represents the probability $p(a,b)$. This model satisfies the assumption of uniformity discussed in Section \ref{S:intro}. The following argument was discovered by the third-named author several years ago \cite{miller}, but has not previously appeared in a formal publication.

\begin{thm}\label{millertime}
The probability $p(a,b)$ that a team with winning percentage $a$ defeats a team with winning
percentage $b$ is given by the James function
\[
P(a,b)=\frac{a(1-b)}{a(1-b)+b(1-a)}\text{,}
\]
except when $a=b=0$ or $a=b=1$, in which case $p(a,b)$ is undefined.
\end{thm}

\begin{proof}  Let teams $A$ and $B$ have winning percentages $a$ and $b$ respectively.
Independently assign to each of $A$ and $B$ either
a $0$ or $1$, where $A$ draws 1 with probability $a$ and $B$ draws $1$ with 
probability $b$.  If one team draws $1$ and the other $0$, the team with $1$ wins the competition.  
If both teams draw the same number, repeat this procedure until they draw different numbers.

The probability that $A$ draws 1 and $B$ draws $0$ on any given turn is clearly $a(1-b)$,
while the opposite occurs with probability $b(1-a)$.
The probability that $A$ and $B$ both draw $1$ is $ab$, and the probability that they
both draw $0$ is $(1-a)(1-b)$.  Hence
\begin{equation}\label{E:total}
ab+(1-a)(1-b)+a(1-b)+b(1-a)=1\text{.}
\end{equation}
It follows that $0\le ab+(1-a)(1-b)\le 1$ and
$0\le a(1-b)+b(1-a)\le 1$ whenever $0\leq a\leq 1$ and $0\leq b\leq 1$.

We can conclude the argument in either of two ways.  Since the probability that $A$ and $B$ draw 
the same number is $ab+(1-a)(1-b)$, in which case they draw again, $p(a,b)$ must 
satisfy the functional equation
\[
p(a,b)=a(1-b)+\left[ab+(1-a)(1-b)\right]p(a,b)\text{.}
\]
The only case in which we cannot solve for $p(a,b)$ is when $ab+(1-a)(1-b)=1$.  By (\ref{E:total}), this situation only 
occurs when $a(1-b)+b(1-a)=0$, which implies that either $a=b=0$ or $a=b=1$.  Otherwise,
$p(a,b)=P(a,b)$.

Alternatively, we may observe that the probability that $A$ wins on the $n$th trial is 
\[
a(1-b)\left[ab+(1-a)(1-b)\right]^{n-1}\text{,}
\]
and so the probability that $A$ wins in at most $n$ trials is
\[
a(1-b)\sum_{k=1}^n\left[ab+(1-a)(1-b)\right]^{k-1}\text{.}
\]
As $n$ tends to $\infty$, this expression yields a convergent geometric series unless $ab+(1-a)(1-b)=1$.
Using (\ref{E:total}), we again obtain the James function.
\end{proof}

This proof relies on a particular model for the relationship between winning percentages and the outcome of a competition. Under different assumptions about this relationship, it seems possible that we would obtain other approximations for $p(a,b)$.  Any such function would presumably also satisfy the James conditions.

\section{Properties of the James function}\label{S:adams}

In this section,  we will consider several important properties of the James function.
We begin by computing the partial derivatives of $P(a,b)$, which will lead to an observation originally due to Dallas
Adams.  Note that
\begin{equation}\label{partial1}
\frac{{\partial}P}{{\partial}a}=\frac{b(1-b)}{\left[a(1-b)+b(1-a)\right]^2}\text{,}
\end{equation}
which shows that the James function satisfies condition (e), and also
\begin{equation}\label{partial2}
\frac{{\partial}P}{{\partial}b}=\frac{-a(1-a)}{\left[a(1-b)+b(1-a)\right]^2}\text{.}
\end{equation}
Furthermore, we have
\[
\frac{{\partial}^2P}{{\partial}a^2}=\frac{-2b(1-b)(1-2b)}{\left[a(1-b)+b(1-a)\right]^3}\text{,}
\]
so that, as a function of $a$, it follows that $P(a,b)$ is concave up for $\frac12<b<1$ and concave
down for $0<b<\frac12$.  Similarly,
\[
\frac{{\partial}^2P}{{\partial}b^2}=\frac{2a(1-a)(1-2a)}{\left[a(1-b)+b(1-a)\right]^3}\text{.}
\]

Adams makes an interesting remark relating to the mixed second partial derivative
\begin{equation}\label{E:mixed}
\frac{{\partial}^2P}{{\partial}a{\partial}b}=\frac{a-b}{\left[a(1-b)+b(1-a)\right]^3}\text{.}
\end{equation}
It follows from (\ref{E:mixed}) that $\frac{{\partial}P}{{\partial}a}$, viewed as a function of $b$, is increasing for
$b<a$ and decreasing for $b>a$, so it is maximized as a function of $b$ when $b=a$.
Since $\frac{{\partial}P}{{\partial}a}$ is positive for every $0<b<1$, it must be most positive when $b=a$.
Alternatively, (\ref{E:mixed}) tells us that $\frac{{\partial}P}{{\partial}b}$, viewed as a function of $a$, is 
increasing for $a>b$ and decreasing for $a<b$, so it is minimized as a function of $a$ when 
$a=b$.  Since $\frac{{\partial}P}{{\partial}b}$ is negative for every $0<a<1$, we conclude that
it is most negative when $a=b$.

Adams interprets these facts in the following manner:  since $P(a,b)$ increases most rapidly with $a$
when $b=a$ (and decreases most rapidly with $b$ when $a=b$), one should field one's
strongest team when playing an opponent of equal strength \cite{james}.  Once again, this observation is perhaps 
more interesting in sports other than baseball, where the star players (other than 
pitchers) play nearly every game when healthy, although James points out that Yankees 
manager Casey Stengel preferred to save his ace pitcher, Whitey Ford, for the strongest opposition.
It seems particularly relevant to European soccer, where the best teams engage in several different
competitions at the same time against opponents of varying quality, and even the top players 
must occasionally be rested.

\bigskip

In principle, there are two ways to increase the value of $P(a,b)$:  by increasing $a$ or by decreasing $b$.
Under most circumstances, a team can only control its own quality and not that of its opponent.  There are some
situations, however, such as the Yankees signing a key player away from the Red Sox,
where an individual or entity might exercise a degree of control over both teams.   Similarly, there are many two-player games
(such as Parcheesi and backgammon) in which each player's move affects the position of both players.  In any such setting,
it is a legitimate question whether the priority of an individual or team should be to improve its own standing or to
diminish that of its adversary.

Recall that the gradient of a function signifies the direction of the greatest rate of increase.
The next result, which has apparently escaped notice until now, follows directly from equations (\ref{partial1}) and (\ref{partial2}).

\begin{prop}\label{gradient}
For any point $(a,b)$, except where $a$ and $b$ both belong to the set $\{0,1\}$, the gradient of the James function $P(a,b)$ is a positive multiple of the vector
\begin{equation*}
\langle b(1-b),-a(1-a)\rangle\text{.}
\end{equation*}
In other words, to maximize the increase of $P(a,b)$, the optimal ratio of the increase of $a$ to the decrease of $b$ is $b(1-b):a(1-a)$.
\end{prop}

One consequence of this result is that when two teams have identical winning percentages,
the optimal strategy for increasing $P(a,b)$ is to increase $a$ and to decrease $b$ in equal measure.
The same fact holds when two teams have complementary winning percentages.
In all other situations, the maximal increase of $P(a,b)$ is achieved by increasing $a$ and decreasing $b$ by different amounts,
with the ratio tilted towards the team whose winning
percentage is further away from $\frac{1}{2}$.  In the extremal cases, when one of the two values $a$ or $b$ belongs
to the set $\{0,1\}$, the optimal strategy
is to devote all resources to changing the winning percentage of the team that is either perfectly good or perfectly bad.   This observation is
somewhat vacuous when $a=1$ or $b=0$, since $P(a,b)$ is already as large as it could possibly be,
although the strategy is entirely reasonable when $a=0$ or $b=1$.  It also makes sense that the
gradient is undefined at the points $(0,0)$, $(0,1)$, $(1,0)$, and $(1,1)$, since these winning percentages do
not provide enough information to determine how much one team must improve to defeat the other.

\bigskip

If $P(a,b)=c$, it is easy to see that $a(1-b)(1-c)=(1-a)bc$, which implies the next result.

\begin{prop}\label{involution}
If $0<a<1$, then $P(a,b)=c$ if and only if $P(a,c)=b$.  In other
words, for a fixed value of $a$, the James function is an involution.
\end{prop}

The practical interpretation of this result is simple to state, even if it is not intuitively obvious:
if team $A$ has probability $c$ of beating a team with winning percentage $b$,
then team $A$ has probability $b$ of beating a team with winning percentage $c$.  The James conditions
already imply this relationship whenever $b$ and $c$ both belong to the set $\{0,1\}$ or the set $\{\frac{1}{2},a\}$.  Nevertheless, it is not
evident at this point whether the involutive property is a necessary consequence of the James conditions.  (Example
\ref{Ex1} will provide an answer to this question.)\bigskip

Proposition \ref{involution} has two further implications that are worth mentioning.  The first is a version of the involutive property
that holds for a fixed value of $b$:
\begin{quote}
If $0<b<1$, then $P(a,b)=1-c$ if and only if $P(c,b)=1-a$.
\end{quote}
The second is that the level curves for the James function (that is, the set of all points for which $P(a,b)=c$ for a particular constant $c$)
can be written
\begin{equation}\label{levelcurve}
b=P(a,c)=\frac{a(1-c)}{a(1-c)+c(1-a)}
\end{equation}
for $0<a<1$.  These level curves are the concrete manifestation of a straightforward principle: if a team $A$ improves by a certain amount,
 there should be a corresponding amount that a team $B$ can improve so that the probability of $A$ defeating
 $B$ remains unchanged.  Each level curve represents the path from $(0,0)$ to $(1,1)$ that such a pair would take in tandem.  (See Figure 1.)
\begin{figure}[h]
\scalebox{.8}{\includegraphics{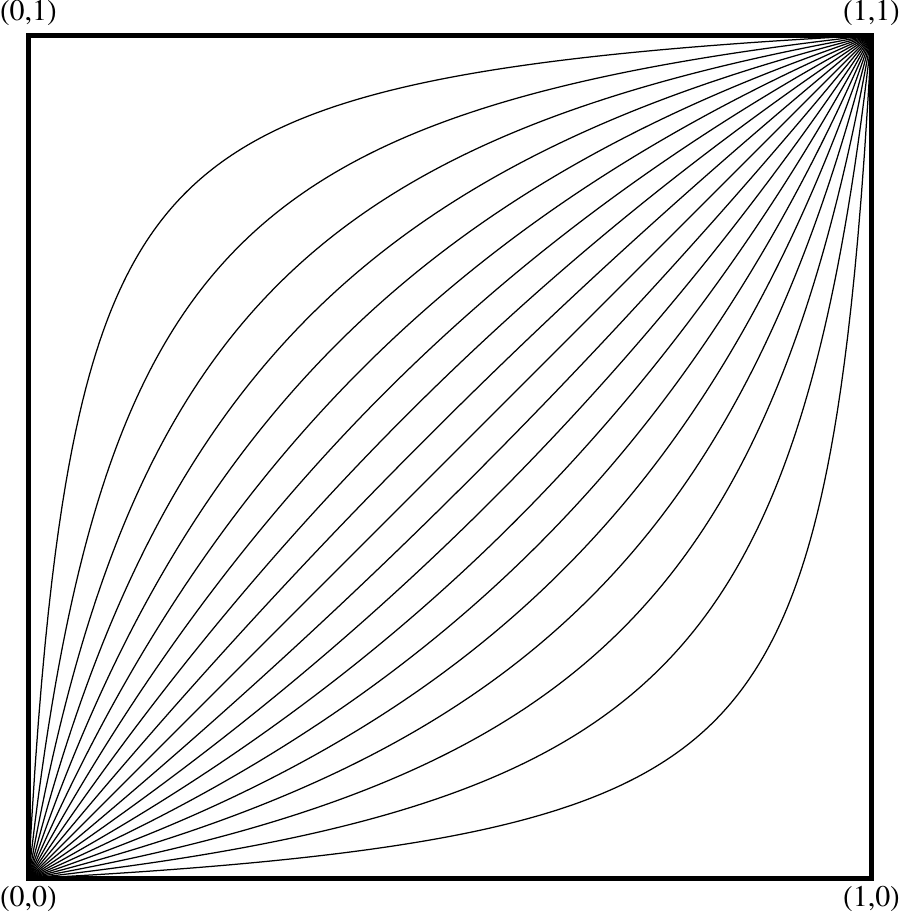}}
\caption{The level curves for the James function $P(a,b)$.}
\end{figure}

We conclude this section with one more observation relating to these level curves.

\begin{prop}\label{diffeq}
For any $0<c<1$, the corresponding level curve for the James function $P(a,b)$ is the unique solution to the differential equation
\[
\frac{db}{da}=\frac{b(1-b)}{a(1-a)}
\]
that passes through the point $(c,\frac{1}{2})$.
\end{prop}

\noindent Another way of stating this result is that, for
two teams to maintain the same value of $P(a,b)$, they should increase (or decrease) their winning percentages according to the ratio $a(1-a):b(1-b)$.
One can either verify this assertion directly, by solving the differential equation to obtain (\ref{levelcurve}), or by
appealing to Proposition \ref{gradient} and recalling
that the gradient is always perpendicular to the level curve at a particular point.

\section{Jamesian functions}\label{S:isa}

We will now consider the question of whether there is a unique function satisfying the James
conditions.  We begin with the following observation, which is implicit in the construction of the James function.

\begin{prop}\label{notbrad}
The James function is the only function derived from the Bradley--Terry model that satisfies the James conditions.
\end{prop}

\begin{proof}
Suppose $\pi(A,B)$ satisfies the James conditions and is derived from the Bradley--Terry model.  Let team $A$ have winning percentage $a$, with $0<a<1$, and let team $C$ have winning percentage $\frac{1}{2}$.
Condition (a) implies that
\[
a=\pi(A,C)=\frac{w(A)}{w(A) + w(C)}\text{.}
\]
Solving for $w(A)$, we obtain
\[
w(A)= \frac{aw(C)}{1-a}= q_{c}(a)\text{,}
\]
where $c = w(C)$.  Thus $\pi(A,B)$ agrees with the James function $P(a,b)$ when both $a$ and $b$
belong to the interval $(0,1)$.  Since the James conditions uniquely determine the value of a function whenever
$a$ or $b$ belongs to $\{0,1\}$, the functions $\pi(A,B)$ and $P(a,b)$ must be identical.
\end{proof}

Let $S$ denote the open unit square $(0,1)\times (0,1)$.  We will say that any function $J(a,b)$,
defined on the set $\overline{S}\setminus\{(0,0)\cup(1,1)\}$, that satisfies the James conditions
is a \textit{Jamesian function}.
Our immediate objective is to disprove the James conjecture
by identifying at least one example of a Jamesian function that is different from the James function $P(a,b)$.
Proposition \ref{notbrad} guarantees that any such function, if it exists, cannot be derived from the Bradley--Terry model.

\begin{ex}\label{Ex1}
We will reverse-engineer our first example of a new Jamesian function by starting with its level
curves.  Consider the family of curves $\{j_{c}\}_{c\in(0,1)}$ defined as follows:
\[
j_{c}(a)=\left\{
\begin{matrix}
\displaystyle\frac{a}{2c}, & 0<a\leq\displaystyle\frac{2c}{1+2c} \\
2ca+1-2c, & \displaystyle\frac{2c}{1+2c}<a<1
\end{matrix}
\right.
\]
for $0<c\leq\frac{1}{2}$ and
\[
j_{c}(a)=\left\{
\begin{matrix}
(2-2c)a, & \displaystyle 0<a\leq\frac{1}{3-2c}\\
\displaystyle\frac{a+1-2c}{2-2c}, & \displaystyle\frac{1}{3-2c}<a<1
\end{matrix}
\right.
\]
for $\frac{1}{2}<c<1$.  (See Figure 2.) These curves have been chosen to satisfy certain symmetry properties,
which the reader can probably deduce but which we will not state explicitly.
(Suffice it to say that $j_{c}(c)=\frac{1}{2}$ for all $c$.)
We define the function $J(a,b)$ on $S$ by assigning to every point $(a,b)$
the value of $c$ associated with the particular curve $j_{c}$ that passes through that point.
We assign the value $0$ or $1$ to points on the boundary of $S$, as dictated by the James conditions.

\begin{figure}[h]
\scalebox{.8}{\includegraphics{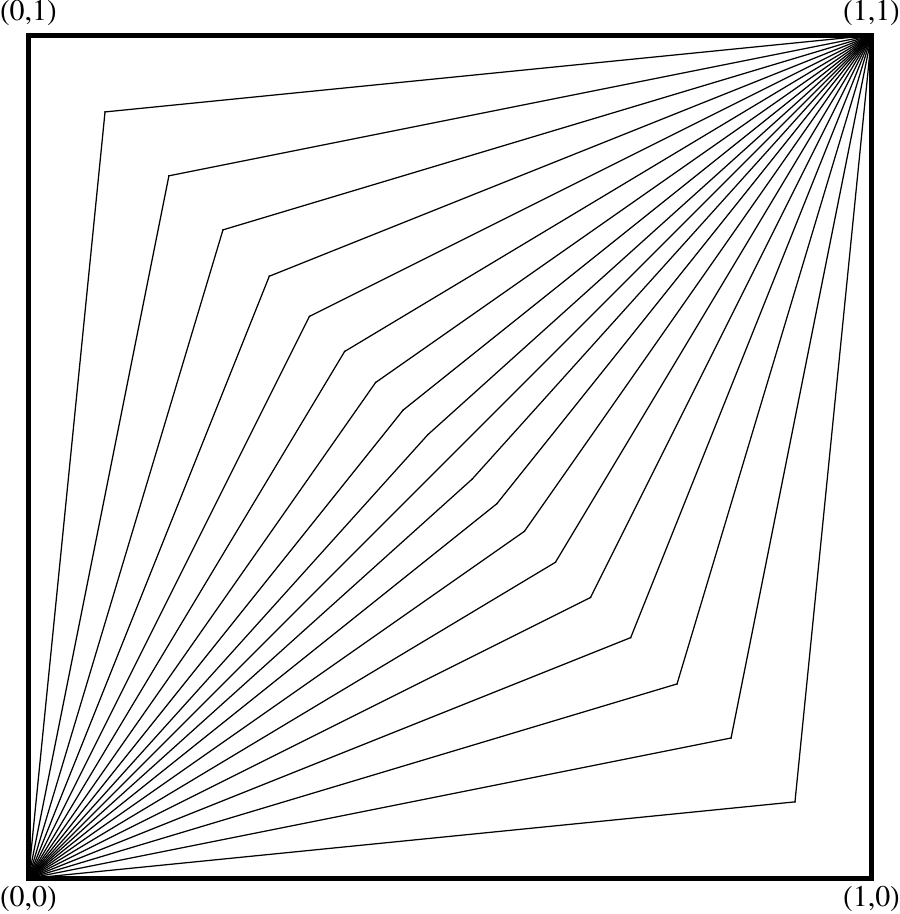}}
\caption{The level curves for the function $J(a,b)$ in Example \ref{Ex1}.}
\end{figure}

A bit more work yields an explicit
formula for $J(a,b)$, from which one can verify directly that all of the James conditions
are satisfied:
\[
J(a,b)=\left\{
\begin{matrix}
\displaystyle\frac{a}{2b}, & (a,b)\in\mathrm{I}\\ \\
\displaystyle\frac{2a-b}{2a}, & (a,b)\in\mathrm{II}\\ \\
\displaystyle\frac{1-b}{2(1-a)}, & (a,b)\in\mathrm{III}\\ \\
\displaystyle\frac{1+a-2b}{2(1-b)}, & (a,b)\in\mathrm{IV}\\
\end{matrix}
\right.\text{,}
\]
where I, II, III, and IV are subsets of $\overline{S}\setminus\{(0,0)\cup(1,1)\}$ that are defined according to Figure 3.

\begin{figure}[h]\label{quadrants}
\scalebox{.8}{\includegraphics{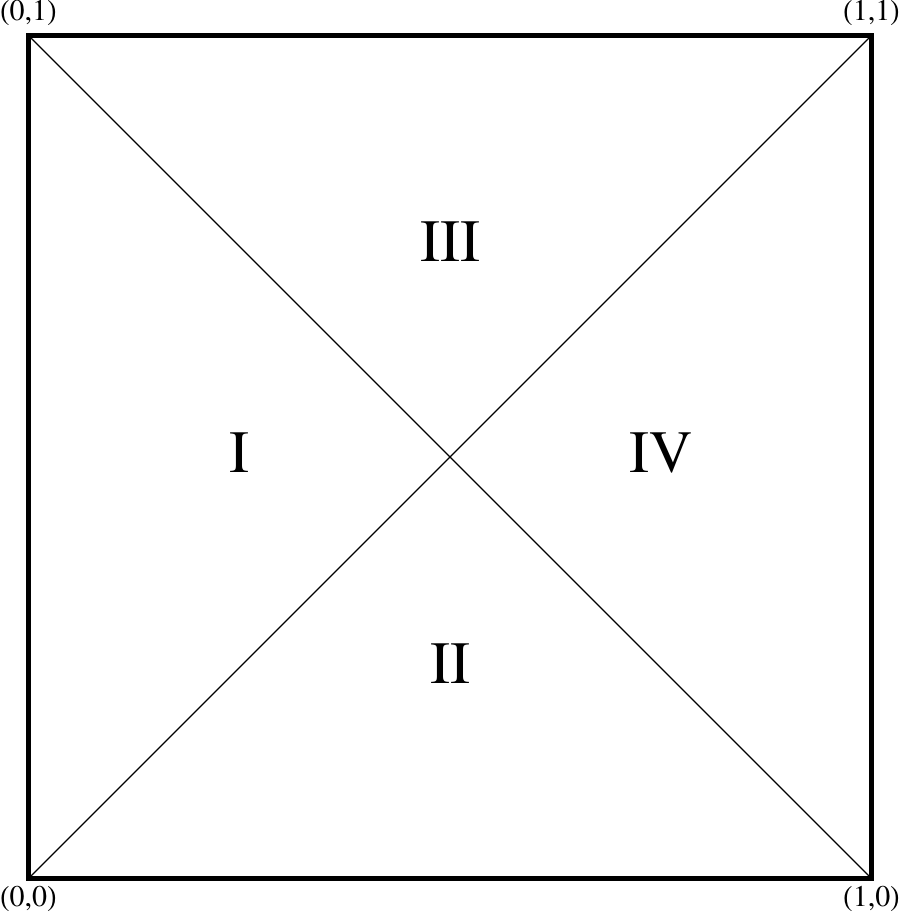}}
\caption{The subsets of $\overline{S}\setminus\{(0,0)\cup(1,1)\}$ in Example \ref{Ex1}.}
\end{figure}

Observe that the appropriate definitions coincide on the boundaries between regions, from which it follows that $J(a,b)$ is
continuous on $\overline{S}\setminus\{(0,0)\cup(1,1)\}$. On the other hand, it is not difficult to see that $J(a,b)$ fails to be differentiable
at all points of the form $(a,1-a)$ for $0<a<\frac{1}{2}$ or $\frac{1}{2}<a<1$.
(With some effort, one can show that it is differentiable at the point $(\frac{1}{2},\frac{1}{2})$.)
In reference to Proposition \ref{involution}, note that $J(\textstyle\frac{1}{3},\frac{1}{4})=\frac{5}{8}$ and
$J(\textstyle\frac{1}{3},\frac{5}{8})=\frac{4}{15}$.  In other words, the involutive property
is not a necessary consequence of the James conditions.
\end{ex}

In view of the preceding example, we need to refine our terminology somewhat.  We will refer to any Jamesian function (such as the James
function itself) that satisfies the condition
\[
J\bigl(a,J(a,b)\bigr)=b
\]
for $0<a<1$ as an \textit{involutive Jamesian function}.
\bigskip

It turns out to be fairly easy to construct Jamesian functions with discontinuities in $S$ (see \cite{supplement}).  Proposition \ref{invcont},
which we will prove in the next section, guarantees that any such function is not involutive.
Rather than considering such pathological examples, we will
devote the next section to examining Jamesian functions that are involutive, continuous, and (in many cases) differentiable.

\section{Involutive Jamesian functions}\label{S:hyp}

We now turn our attention to Jamesian functions that satisfy the involutive property
\[
J\bigl(a,J(a,b)\bigr)=b\text{,}
\]
or equivalently
\[
J(a,b)=c\text{ if and only if }J(a,c)=b\text{,}
\]
whenever $0<a<1$.  This property essentially subsumes three of the five James conditions (namely (a), (b), and (d)).

\begin{prop}
A function $J\colon\overline{S}\setminus\{(0,0)\cup(1,1)\}\rightarrow[0,1]$ is an involutive Jamesian function if and only if it satisfies the involutive property,
James condition (c), and James condition (e).
\end{prop}

\begin{proof}  By definition, an involutive Jamesian function must satisfy the involutive property, as well as all five James conditions.  Suppose then that $J(a,b)$
satisfies the involutive property, together with conditions (c) and (e).

To see that $J(a,b)$ satisfies condition (b), take $0<a<1$ and suppose that $J(a,0)=c$ for $0\leq c<1$.  The involutive property would then dictate that $J(a,c)=0$, and thus condition (c) would imply that $J(c,a)=1$.  Hence $J(c^{\prime},a)\leq J(c,a)$ for $c<c^{\prime}\leq 1$, which would violate condition (e).  Consequently $J(a,0)=1$ for $0<a<1$.
Since $J(a,b)$ is a non-decreasing function of $a$, we conclude that $J(1,0)=1$ as well.

Next consider condition (d).  Applying the involutive property three times and condition (c) twice, we see that
\begin{align*}
J(a,b)=c\hspace{.1in}
\Longleftrightarrow\hspace{.1in}&J(a,c)=b\\
\Longleftrightarrow\hspace{.1in}&J(c,a)=1-b\\
\Longleftrightarrow\hspace{.1in}&J(c,1-b)=a\\
\Longleftrightarrow\hspace{.1in}&J(1-b,c)=1-a\\
\Longleftrightarrow\hspace{.1in}&J(1-b,1-a)=c\text{,}
\end{align*}
as long as $a$, $b$, and $c$ all belong to the interval $(0,1)$.  The cases where $a$, $b$, or $c$ belongs to $\{0,1\}$ can be dealt with
by appealing to condition (b).  In particular, we know that $J(a,0)=1$ for $0<a\leq 1$, which implies that $J(1-a,0)=1$ for $0\leq a<1$.
The involutive property dictates that $J(1-a,1)=0$
for $0<a<1$.  Since $J(1,0)=1$, it follows from (c) that $J(1,1-a)=1=J(a,0)$ for $0<a\leq1$.   Hence
condition (d) holds whenever $b=0$.  The remaining cases can be deduced from this observation.

Finally, consider condition (a).  Taking $b=a$ in condition (c), we see that $J(a,a)=\frac{1}{2}$.  Hence the involutive property dictates that $J(a,\frac{1}{2})=a$ for $0<a<1$.
For $a=1$, simply note that conditions (d) and (b) imply that $J(1,\frac{1}{2})=J(\frac{1}{2},0)=~1$.  Similarly, condition (c) shows that $J(0,\frac{1}{2})=1-J(\frac{1}{2},0)=0$.
\end{proof}

\noindent In other words, to identify an involutive Jamesian function, we can restrict our attention to the following set of conditions:
\begin{enumerate}[label=(\roman*)]
\item $J\bigl(a,J(a,b)\bigr)=b$ for $0<a<1$.
\item $J(b,a)=1-J(a,b)$.
\item $J(a,b)$ is a non-decreasing function of $a$ for $0\leq b\leq 1$ and a strictly increasing function of $a$ for $0<b<1$.
\end{enumerate}
We will refer to this list as the \textit{involutive James conditions}.

Condition (i) also guarantees that a Jamesian function possesses another important property.

\begin{prop}\label{invcont}
Every involutive Jamesian function is continuous on $\overline{S}\setminus\{(0,0)\cup(1,1)\}$.
\end{prop}

\begin{proof}
Take a fixed value $0<c<1$ and consider the level curve  $J(a,b)=c$, which can be
rewritten $b=J(a,c)$ for $0<a<1$.  Conditions (i) and (ii) imply that
\[
J\bigl(1-J(a,c),c\bigr)=1-a\text{.}
\]
Thus $J(a,c)$, viewed as a function of $a$, is a bijection from the interval $(0,1)$ onto itself.
Hence it follows from (iii) that the curve $J(a,c)$ is a continuous, strictly increasing function of $a$ that connects the points $(0,0)$ and $(1,1)$.

Suppose, for the sake of contradiction, that $J(a,b)$ fails to be continuous at a point $(a_{0},b_{0})$ in $S$.
In other words, there exists a positive number $\varepsilon_{0}$ such that, for any positive $\delta$, there is a point $(a,b)$ such that $\|(a,b)-(a_{0},b_{0})\|<\delta$
and $|J(a,b)-J(a_{0},b_{0})|\geq\varepsilon_{0}$.  (If necessary, redefine $\varepsilon_{0}$ so it is less than
$\min\{2J(a_{0},b_{0}),2-2J(a_{0},b_{0})\}$.)
Let $c_{1}=J(a_{0},b_{0})-\varepsilon_{0}/2$
and $c_{2}=J(a_{0},b_{0})+\varepsilon_{0}/2$, and consider the level curves $J(a,c_{1})$ and $J(a,c_{2})$.  Let $\delta_{0}$
denote the minimum of the distance between $(a_{0},b_{0})$ and $J(a,c_{1})$ and the distance between $(a_{0},b_{0})$ and
$J(a,c_{2})$.

By assumption, there is a point $(a_{3},b_{3})$ such that $\|(a_{3},b_{3})-(a_{0},b_{0})\|<\delta_{0}$
and $c_{3}=J(a_{3},b_{3})$ is either less than or equal to  $J(a_{0},b_{0})-\varepsilon_{0}$ or greater than or equal to
 $J(a_{0},b_{0})+\varepsilon_{0}$.  Since $J(a,c_{i})=\frac{1}{2}$ at $a=c_{i}$, the level curve $J(a,c_{3})$ intersects the line $b=\frac{1}{2}$ either to the left of
 the curve $J(a,c_{1})$ or to the right of the curve $J(a,c_{2})$.  On the other hand, since $(a_{3},b_{3})$ lies within $\delta_{0}$ of $(a_{0},b_{0})$,
the curve $J(a,c_{3})$ must intersect the line $b=b_{3}$ between $J(a,c_{1})$ and $J(a,c_{2})$.  Hence two of the level curves must intersect at a point in $S$, which is impossible.  (See Figure 4 for a graphical illustration of this argument.)

Now consider a point $(a_{0},b_{0})$ on the boundary of $S$.  The only difference in the proof is that, if $a=0$ or $b=1$, the level curve $J(a,c_{1})$ does not exist.  In this case, it is not difficult to see that $J(a,c_{3})$ must intersect the curve $J(a,c_{2})$.  Similarly, if $a=1$ or $b=0$,
there is no level curve $J(a,c_{2})$, but one can show that $J(a,c_{3})$ must intersect $J(a,c_{1})$.
\end{proof}

\begin{figure}[h]\label{contfigure}
\scalebox{.8}{\includegraphics{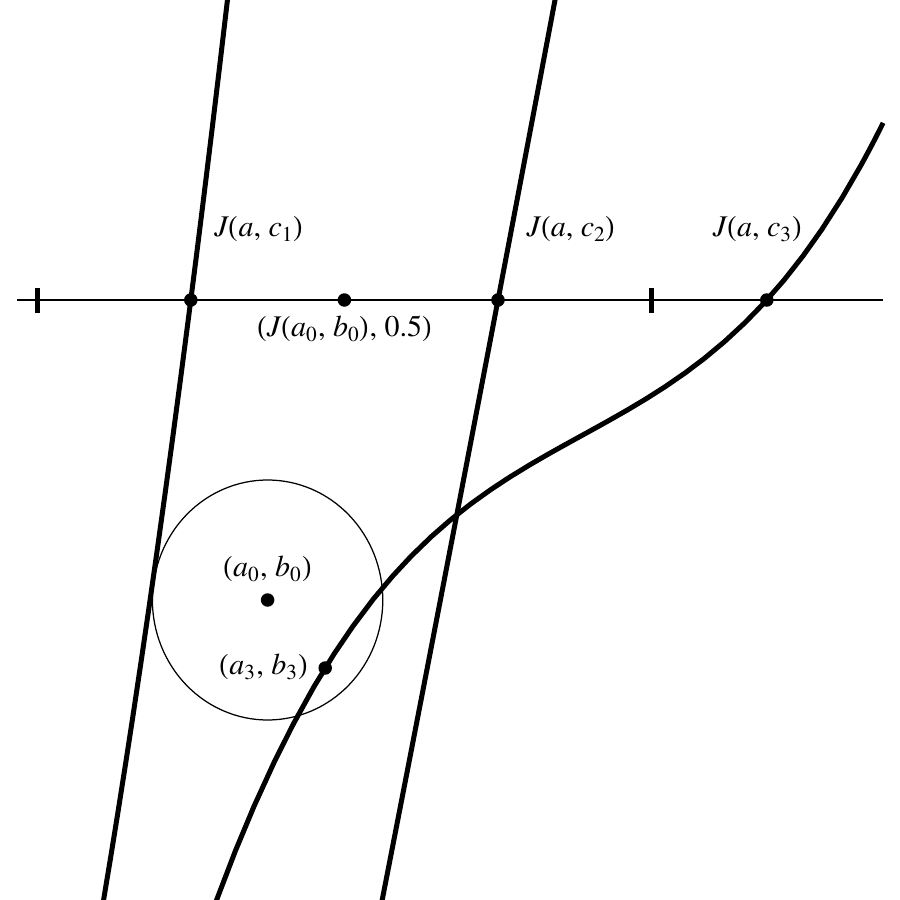}}
\caption{An illustration of the proof of Proposition \ref{invcont}.}
\end{figure}
\bigskip

Let $g\colon (0,1)\rightarrow\mathbb{R}$ be a continuous, strictly increasing function that satisfies the conditions
\begin{itemize}
\item $g(1-a)=-g(a)$.
\item $\displaystyle\lim_{a\rightarrow 0^{+}}g(a)=-\infty$.
\end{itemize}
These conditions imply that $g(\frac{1}{2})=0$ and that 
\[
\lim_{a\rightarrow 1^{-}}g(a)=\infty\text{.}
\]
Observe that $g^{-1}\colon\mathbb{R}\rightarrow(0,1)$ is
a continuous, strictly increasing function with $g^{-1}(-s)=1-g^{-1}(s)$.  It makes sense to define $g(0)=-\infty$
and $g(1)=\infty$, so that $g^{-1}(-\infty)=0$ and $g^{-1}(\infty)=1$.
We claim that any such function $g$ can be used to construct an involutive Jamesian function.

\begin{thm}\label{jacthm}
For any $g$ satisfying the conditions specified above, the function
\begin{equation}\label{ginvg}
J(a,b)=g^{-1}\bigl(g(a)-g(b)\bigr)
\end{equation}
is an involutive Jamesian function.
\end{thm}

\begin{proof}
Consider each of the three involutive James conditions:

(i) Note that
 \begin{align*}
J\bigl(a,J(a,b)\bigr)&=g^{-1}\bigl(g(a)-g\bigl(g^{-1}\bigl(g(a)-g(b)\bigr)\bigr)\bigr)\\
&=g^{-1}\bigl(g(a)-g(a)+g(b)\bigr)\\
&=g^{-1}\bigl(g(b)\bigr)=b\text{,}
\end{align*}
as long as $0<a<1$.  (The cases where $a=0$ and $a=1$ yield the indeterminate forms $-\infty+\infty$ and $\infty-\infty$.)

(ii) Similarly,
\[
J(b,a)=g^{-1}\bigl(g(b)-g(a)\bigr)=1-g^{-1}\bigl(g(a)-g(b)\bigr)=1-J(a,b)\text{.}
\]

(iii) Since both $g$ and $g^{-1}$ are strictly increasing, it follows that $J(a,b)$ is a strictly increasing function of $a$ when $0<b<1$.
Moreover, $J(a,b)$ takes on the constant value $1$ when $b=0$ and the constant value $0$ when $b=1$.
\end{proof}

\noindent While it is unnecessary to verify James conditions (a) and (d), it is worth noting that (a) corresponds to the property $g(\frac{1}{2})=0$ and (d) to the property
$g(1-a)=-g(a)$.  In effect, we verified condition (b) in the process of considering (iii).

It is easy to use Theorem \ref{jacthm} to generate concrete examples.

\begin{ex}\label{Ex2}
The function
\[
g(a)=\frac{2a-1}{a(1-a)}
\]
satisfies all the necessary conditions for Theorem \ref{jacthm}, so (\ref{ginvg}) defines an
involutive Jamesian function.  Since
\[
g^{-1}(s)=\frac{s-2+\sqrt{s^{2}+4}}{2s}\text{,}
\]
we obtain
\[
J(a,b)=\frac{x+y-\sqrt{x^2+y^2}}{2y}=\frac{x}{x+y+\sqrt{x^2+y^2}}\text{,}
\]
where $x=2ab(1-a)(1-b)$ and $y=(b-a)(2ab-a-b+1)$.
\end{ex}

\begin{ex}\label{Ex3}
The function $g(a)=-\cot(\pi a)$ yields the involutive Jamesian function
\[
J(a,b)=\frac{1}{\pi}\cot^{-1}\bigl(\cot(\pi a)-\cot(\pi b)\bigr)\text{,}
\]
where we are using the version of the inverse cotangent  that attains values between
$0$ and $\pi$.
\end{ex}
\bigskip

The construction described in Theorem \ref{jacthm} is closely related to what is known as a \textit{linear
model} for paired comparisons.  In such a model,
\[
\pi(A,B)=F\bigl(v(A)-v(B)\bigr)\text{,}
\]
where $v$ denotes a measure of worth and $F$ is the cumulative distribution function of a random
variable that is symmetrically distributed about $0$ (see \cite[Section 1.3]{david}).  The Bradley--Terry model can be viewed
as a linear model, where $F$ is the logistic function
\[
F(s)=\frac{e^{s}}{e^{s}+1}=\int_{-\infty}^{s}\frac{e^{t}}{(1+e^{t})^{2}}dt
\]
and $v(A)=\log w(A)$.  In particular, the James function can be constructed in the manner of Theorem \ref{jacthm}, with
$F=g^{-1}$ being the logistic function and $g$ being the so-called logit function
\[
g(a)=\log\!\left(\frac{a}{1-a}\right)\text{.}
\]
(This observation could charitably be construed as an \textit{a posteriori} justification for the term ``log5" originally
used by James.)

What is distinctive about the James function in this context is that the construction
is symmetric, with $v(A)=\log w(A)$ and $v(B)=\log w(B)$ replaced by $g(a)=\log(a/(1-a))$ and $g(b)=\log(b/(1-b))$ respectively.
This symmetry corresponds to the twofold application of the Bradley--Terry model that was discussed in Section \ref{S:intro}.
Likewise, the fact that both $g$ and $g^{-1}$ appear in the general formulation of Theorem \ref{jacthm} can be interpreted as a consequence of the same
model being used to define both worth and probability.

\bigskip

\begin{ex}\label{Ex4}
Take
\[
F(s)=g^{-1}(s)=\frac{1}{\sqrt{2\pi}}\int_{-\infty}^{s}e^{-\frac{t^{2}}{2}}dt\text{,}
\]
so that $g$ is the so-called probit function.  The involutive Jamesian function $J(a,b)=g^{-1}\bigl(g(a)-g(b)\bigr)$ can be considered
the analogue of the James function relative to the Thurstone--Mosteller model (see \cite{david}).
\end{ex}

\bigskip

Theorem \ref{jacthm} allows us to identify a large class of functions that can be viewed as generalizations of the James function.  Since
\[
\log\!\left(\frac{a}{1-a}\right)=\int_{\frac{1}{2}}^{a}\left(\frac{1}{t}+\frac{1}{1-t}\right)dt=\int_{\frac{1}{2}}^{a}\frac{1}{t(1-t)}dt\text{,}
\]
we define
\[
g_{n}(a)=\int_{\frac{1}{2}}^{a}\frac{1}{(t(1-t))^{n}}dt
\]
for any real number $n\geq 1$.  It is not difficult to verify that $g_{n}$ satisfies all of the prescribed requirements for Theorem \ref{jacthm}.  (The stipulation that $g_{n}(0)=-\infty$ precludes the case where $0<n<1$.)
Define
\begin{equation}\label{hypjameq}
H_{n}(a,b)=g_{n}^{-1}\bigl(g_{n}(a)-g_{n}(b)\bigr)\text{.}
\end{equation}
For $n>1$, we shall refer to $H_{n}(a,b)$ as a \textit{hyper-James function}.  Each of these functions
is an involutive Jamesian function.

In some situations, it is possible to obtain a more concrete representation for
$H_{n}(a,b)$.  For example, one can show that
\[
g_{\frac{3}{2}}(a)=\frac{2(2a-1)}{\sqrt{a(1-a)}}
\]
and
\[
g_{\frac{3}{2}}^{-1}(s)=\frac{s+\sqrt{s^{2}+16}}{2\sqrt{s^{2}+16}}\text{,}
\]
and hence
\[
H_{\frac32}(a,b)=\frac12+\frac{v^{\prime}\sqrt u-u^{\prime}\sqrt v}{2\sqrt{u+v-4uv-2u^{\prime}v^{\prime}\sqrt{uv}}}
\]
for $u=a(1-a)$, $v=b(1-b)$, $u^{\prime}=1-2a$, and $v^{\prime}=1-2b$ (see \cite{supplement} for more details).  In general, though, it seems unlikely that there is an explicit formula for $H_{n}(a,b)$ that is more
useful than (\ref{hypjameq}).

\bigskip

We will now examine the issue of differentiability.
For any function defined according to Theorem \ref{jacthm}, a routine calculation shows that
\begin{equation}\label{djda}
\frac{\partial J}{\partial a}=\frac{g^{\prime}(a)}{g^{\prime}\bigl(J(a,b)\bigl)}
\end{equation}
and
\begin{equation}\label{djdb}
\frac{\partial J}{\partial b}=\frac{-g^{\prime}(b)}{g^{\prime}\bigl(J(a,b)\bigl)}
\end{equation}
at all points $(a,b)$ for which the above quotients are defined.  Based on this observation, we are able to
obtain the following result.

\begin{prop}\label{diffprop}
If $g$ is continuously differentiable on $(0,1)$, with $g^{\prime}$ never equal to $0$, the corresponding Jamesian function $J(a,b)$ is differentiable on $S$.
Conversely,  if $J(a,b)$ is differentiable on $S$, the function $g$ must be differentiable on $(0,1)$ with $g^{\prime}$ never $0$.
\end{prop}

\begin{proof}
Suppose that $g^{\prime}$ is continuous and nonzero on $(0,1)$.
It follows from (\ref{djda}) and (\ref{djdb}) that
both $\frac{\partial J}{\partial a}$ and $\frac{\partial J}{\partial b}$ are defined and continuous at
all points in the open set $S$, which guarantees that $J(a,b)$ is differentiable on $S$.

Now suppose that $J(a,b)$ is differentiable at every point in $S$.  Let $a_{0}$ be an arbitrary element of $(0,1)$.
Since $g$ is strictly increasing, it could only fail to be differentiable on a set of measure $0$ (see \cite[p.\ 112]{royden}).
In particular, there is at least one  $c$ in $(0,1)$ for which $g^{\prime}(c)$ is defined.
Since $J(a_{0},b)$, viewed as a function of $b$, attains every value in the interval $(0,1)$,
there exists a $b_{0}$ in $(0,1)$ such that $J(a_{0},b_{0})=c$.
Note that
\[
g(a)=g\bigl(J(a,b_{0})\bigr)+g(b_{0})
\]
for all $a$ in $(0,1)$, so the chain rule dictates that
\[
g^{\prime}(a_{0})=g^{\prime}(c)\cdot\frac{\partial J}{\partial a}(a_{0},b_{0})\text{.}
\]
Therefore $g$ is differentiable on the entire interval $(0,1)$.  Suppose, for the sake of contradiction, that
there were some $d$ in $(0,1)$ for which $g^{\prime}(d)=0$.
As before,  there would exist a $b_{1}$ in $(0,1)$ such that $J(a_{0},b_{1})=d$, which would imply that
\[
g^{\prime}(a_{0})=g^{\prime}(d)\cdot\frac{\partial J}{\partial a}(a_{0},b_{1})=0\text{.}
\]
Consequently $g^{\prime}$ would be identically $0$ on $(0,1)$, which is impossible.
\end{proof}

In other words, all the specific examples of Jamesian functions we have introduced in this section,
including the hyper-James functions,
are differentiable on $S$.  We can now state a more general version of Proposition \ref{gradient},
which follows directly from (\ref{djda}) and (\ref{djdb}).

\begin{prop}\label{jacgrad}
For any differentiable Jamesian function $J(a,b)$ defined according to Theorem \ref{jacthm}, the gradient
at a point $(a,b)$ in $S$ is a positive multiple of the vector $\langle g^{\prime}(a),-g^{\prime}(b)\rangle$.
\end{prop}

If $g$ is differentiable on $(0,1)$, the condition that $g(1-a)=-g(a)$ implies that $g^{\prime}(1-a)=g^{\prime}(a)$.  Hence the gradient of
$J(a,b)$ is a positive multiple
of $\langle 1,-1\rangle$ whenever $b=a$ or $b=1-a$.  This observation generalizes the fact that,
whenever two teams have identical or complementary winning percentages,
the optimal strategy for increasing $P(a,b)$ is to increase $a$ and decrease $b$ by equal amounts.

\bigskip

For any Jamesian function given by (\ref{ginvg}), the level curve $J(a,b)=c$ for $0<c<1$ can be rewritten
\[
b=J(a,c)=g^{-1}\bigl(g(a)-g(c)\bigr)\text{,}
\]
or $g(a)=g(b)+g(c)$.  Hence we have the following generalization of Proposition \ref{diffeq}.

\begin{prop}
Let $J(a,b)$ be a differentiable Jamesian function defined according to Theorem \ref{jacthm}.  For any $0<c<1$,
the corresponding level curve for $J(a,b)$ is the unique solution to the differential equation
\[
\frac{db}{da}=\frac{g^{\prime}(a)}{g^{\prime}(b)}
\]
that passes through the point $(c,\frac{1}{2})$.
\end{prop}

Thus the level curves for the Jamesian functions defined in Examples \ref{Ex2} and \ref{Ex3} are given by the differential equations
\[
\frac{db}{da}=\frac{(2a^{2}-2a+1)(b(1-b))^2}{(2b^{2}-2b+1)(a(1-a))^{2}}
\]
and
\[
\frac{db}{da}=\left(\frac{\sin(\pi b)}{\sin(\pi a)}\right)^{2}
\]
respectively.  Likewise, the level curves for any hyper-James function $H_{n}(a,b)$ are given by the differential equation
\[
\frac{db}{da}=\left(\frac{b(1-b)}{a(1-a)}\right)^{n}\text{.}
\]
Figure 5 shows the level curves for the hyper-James function $H_{2}(a,b)$.

\begin{figure}[h]
\scalebox{.8}{\includegraphics{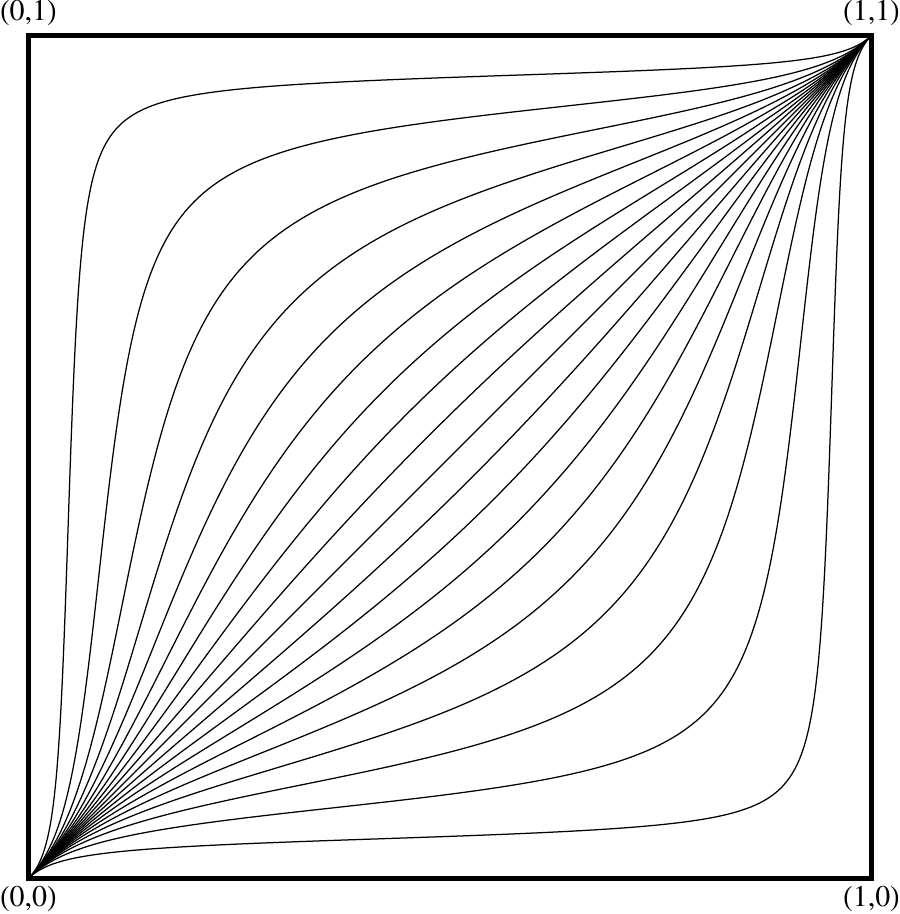}}
\caption{The level curves for the hyper-James function $H_{2}(a,b)$.}
\end{figure}

\section{Final thoughts}

While it is possible to construct additional examples of non-involutive Jamesian functions,
it would be reasonable to focus any further investigation on the involutive case.
Perhaps the most obvious question is whether one can assign any probabilistic significance to the involutive
Jamesian functions we have just introduced, particularly the hyper-James functions.
For instance, could one somehow alter the assumptions underlying Theorem \ref{millertime}
to obtain one of these functions in place of $P(a,b)$?

Within this context, several lines of inquiry seem especially worthwhile:
\begin{enumerate}
\item Does every involutive Jamesian function have the form described in Theorem \ref{jacthm}, for some particular
function $g$?
\item  While it is clear how the involutive property arises mathematically, is there any \textit{a priori} reason
that it should hold, based on the probabilistic interpretation of the James function?
\item Are there any situations for which non-differentiability would
make sense in the setting of an athletic competition?
\end{enumerate}
We would be delighted if this paper motivated other mathematicians (or sports enthusiasts) to
consider any of these questions.

\section*{Acknowledgments}\label{S:ack}

We would never have written this paper if Caleb Garza, a recent alumnus of
Connecticut College, had not decided
to give a senior seminar talk on a topic from sabermetrics.  We are sincerely grateful to him for prompting (or reviving) our
interest in this material and for bringing the work of the third-named author to the attention of the first two.  We would also like
to thank the referees and editors of this paper for providing substantial assistance and guidance.

\renewcommand{\abstractname}{Summary}
 
 \begin{abstract} We investigate the properties of the James function, associated
with Bill James's so-called ``log5 method," which assigns a probability to the
result of a game between two teams based on their respective winning percentages.
We also introduce and study a class of functions,
which we call \textit{Jamesian}, that satisfy the same \textit{a priori} conditions that
were originally used to describe the James function. 
\end{abstract}
 
 \end{document}